\documentclass[12pt,reqno]{article}
\usepackage{amsmath,amssymb,amsthm}
\usepackage{color}
\def\R{\mathbb R}  \def\C{\mathbb C} \def\N{\mathbb N}
\newtheorem{thm}{Theorem}[section]
\newtheorem{lem}[thm]{Lemma}
\newtheorem{cor}[thm]{Corollary}
\newtheorem{defn}[thm]{Definition}

\newtheorem{rem}[thm]{Remark}
\numberwithin{equation}{section}
\title{Small data global well--posedness and scattering for the inhomogeneous nonlinear Schr\"{o}dinger equation in $H^{s} (\R^{n})$}
\author{{\bf JinMyong An, JinMyong Kim$^*$}\\
\footnotesize{Faculty of Mathematics, {\bf Kim Il Sung} University, Pyongyang, Democratic People's Republic of Korea}\\
\footnotesize{$^*$ Corresponding Author}\\
\footnotesize{Email address: jm.kim0211@ryongnamsan.edu.kp}
}
\date{}

\begin{document}
\maketitle
\begin{abstract}
We consider the Cauchy problem for the inhomogeneous nonlinear Schr\"{o}dinger (INLS) equation
\[iu_{t} +\Delta u=|x|^{-b} f\left(u\right), u\left(0\right)=u_{0} \in H^{s} (\R^{n}),\]
where $0<s<\min \left\{n,\;\frac{n}{2} +1\right\}$, $0<b<\min \left\{2,\;n-s,{\rm \; 1}+\frac{n-2s}{2} \right\}$ and $f\left(u\right)$ is a nonlinear function that behaves like $\lambda \left|u\right|^{\sigma } u$ with $\lambda \in \C$ and $\sigma >0$. We prove that the Cauchy problem of the INLS equation is globally well--posed in $H^{s} (\R^{n})$ if the initial data is sufficiently small and $\sigma _{0} <\sigma <\sigma _{s} $, where $\sigma _{0} =\frac{4-2b}{n} $ and $\sigma _{s} =\frac{4-2b}{n-2s} $ if $s<\frac{n}{2} $; $\sigma _{s} =\infty $ if $s\ge \frac{n}{2} $. Our global well--posedness result improves the one of Guzm\'{a}n in (Nonlinear Anal. Real World Appl. 37: 249--286, 2017) by extending the validity of $s$ and $b$. In addition, we also have the small data scattering result.
\end{abstract}

\section{Introduction}

In this paper, we study the Cauchy problem for the inhomogeneous nonlinear Schr\"{o}dinger (INLS) equation
\begin{equation} \label{GrindEQ__1_1_}
\left\{\begin{array}{l} {iu_{t} +\Delta u=|x|^{-b} f{\rm (}u{\rm ),}} \\ {u(0,\; x)=u_{0} (x),} \end{array}\right.
\end{equation}
where $u_{0} \in H^{s} (\R^{n})$, $0<s<\min \left\{n,\;\frac{n}{2} +1\right\}$ and $0<b<\min \left\{{\rm 2,\; }n-s,{\rm \; 1}+\frac{n-2s}{2} \right\}$ and $f$ is of class $X\left(\sigma,s,b\right)$ (see Definition 1.1).

Recall that (1.1) has the following equivalent form:
\begin{equation} \label{GrindEQ__1_2_}
u\left(t\right)=S\left(t\right)u_{0} -i\int _{0}^{t}S\left(t-\tau \right)|x|^{-b} f\left(u\left(\tau \right)\right)d\tau  ,
\end{equation}
where $S(t)=e^{it\Delta } $ is the Schr\"{o}dinger semi--group.

First of all, let us give the definition of class $X\left(\sigma,s,b\right)$.

\begin{defn}[\cite{AK21}]\label{defn 1.1}
\textnormal{Let $f:\C\to \C$, $s\ge 0$, $\sigma >0$, $0\le b<2$ and $\left\lceil s\right\rceil $ denote the minimal integer which is larger than or equals to $s$. For $k\in \N$, let $k$--th order derivative of $f(z)$ be defined under the identification $\C=\R^{2}$ $($see Section 2$)$.
We say that $f$ is of class $X\left(\sigma,s,b\right)$ if it satisfies one of the following conditions:}

\textnormal{$\cdot f\left(z\right)$ is a polynomial in $z$ and $\bar{z}$ satisfying that $1<\deg \left(f\right)=1+\sigma \le 1+\frac{4-2b}{n-2s},$ if $s<\frac{n}{2}$ and that $1<\deg \left(f\right)=1+\sigma <\infty $, if $s\ge \frac{n}{2}.$}

\textnormal{$\cdot f\in C^{\max\left\{\left\lceil s\right\rceil, 1 \right\}} \left(\C\to \C\right)$ and
\begin{equation} \label{GrindEQ__1_3_}
\left|f^{(k)}(z)\right|\lesssim\left|z\right|^{\sigma +1-k},
\end{equation}
for any $0\le k\le \max\left\{\left\lceil s\right\rceil, 1 \right\}$ and $z\in \C$, where we assume that $\left\lceil s\right\rceil-1 \le\sigma\le\frac{4-2b}{n-2s},$ if $s<\frac{n}{2}$ and that $\left\lceil s\right\rceil-1 \le\sigma <\infty $, if $s\ge \frac{n}{2}$.}
\end{defn}

\begin{rem}\label{rem1.2}
\textnormal{Let $s\ge 0$ and $0<b<2$. Assume that $0<\sigma \le \frac{4-2b}{n-2s} $, if $s<\frac{n}{2} $, and that $0<\sigma <\infty $, if $s\ge \frac{n}{2} $. If $\sigma $ is not an even integer, assume further $\left\lceil s\right\rceil \le\sigma +1$. Then we can easily verify that $f\left(u\right)=\lambda \left|u\right|^{\sigma } u$ with $\lambda \in \C$ is a model case of class $X\left(\sigma,s,b\right)$. See also \cite{DYC13,W04,WHHG11} for example.}
\end{rem}
When $b=0$, the equation (1.1) is the classic nonlinear Schr\"{o}dinger equation which has been widely studied over the last three decades. On the other hand, at the end of the last century, the inhomogeneous nonlinear Schr\"{o}dinger equation was suggested for modeling the propagation of laser beam in some situations, and it is of a form:
\begin{equation} \label{GrindEQ__1_4_}
iu_{t} +\Delta u+V\left(x\right)\left|u\right|^{\sigma } u=0.
\end{equation}

 For the physical background of (1.4), we can see \cite{G00, G17, LT94} and the references therein. Eq. (1.4) has been studied by several authors over the last two decades. For example, Merle \cite{M96} and Rapha\"el-Szeftel \cite{RS11} studied (1.4) assuming $k_{1} <V\left(x\right)<k_{2} $ with $k_{1} ,\;k_{2} >0$. Fibich-Wang \cite{FW03} studied (1.4) with $V\left(x\right)=V\left(\varepsilon \left|x\right|\right)$ where $\varepsilon >0$ is small and $V\in C^{4} (\R^{n})\bigcap L^{\infty } (\R^{n})$. The case $V\left(x\right)=\left|x\right|^{b} $ with $b>0$ was also studied by many authors (see e.g. \cite{C10,LWW06,Z14} and the references therein).

In this paper, we are interested in $V\left(x\right)=|x|^{-b} $ with $b>0$, i.e. we study the INLS equation (1.1). The INLS equation (1.1) has also attracted a lot of attention in recent years. We refer the reader to [1, 2, 5, 7--10, 12, 13, 15] for recent work on (1.1).

Before recalling the known results for the INLS equation (1.1), we define the following important numbers which are used throughout the paper:
\begin{equation} \label{GrindEQ__1_5_}
\sigma _{s} =\left\{\begin{array}{l} {\frac{4-2b}{n-2s} ,\;0\le s<\frac{n}{2} ,} \\ {\infty ,\;s\ge \frac{n}{2} .} \end{array}\right.
\end{equation}
\begin{equation} \label{GrindEQ__1_6_}
\hat{2}=\left\{\begin{array}{l} {\min \left\{2,{\rm \; 1}+\frac{n-2s}{2} \right\},\;n\ge 3,} \\ {n-s,\;n=1,\;2.} \end{array}\right.
\end{equation}
When $0\le s<\frac{n}{2} $, $\sigma _{s} $ is said be a critical power in $H^{s} (\R^{n})$ (see \cite{AK21,G17}). If $s\ge 0$, we say that $\sigma <\sigma _{s} $ is a subcritical power in $H^{s} (\R^{n})$.

Now let us recall the known well--posdness results for the INLS equation. Cazenave \cite{C03} studied the local and global well--posedness in
$H^{1} (\R^{n})$. Using an abstract theory, he proved that it is
appropriate to seek solution of (1.1) with $f\left(u\right)=\lambda \left|u\right|^{\sigma } u$ satisfying
\[u\in C\left(\left[0,{\rm \; }T\right){\rm ,\; }H^{1} \left(\R^{n}
\right)\right)\bigcap C^{1} \left(\left[0,{\rm \; }T\right){\rm ,\; }H^{-1}
(\R^{n})\right)\]
for some $T>0$. He also proved that any local solution of the INLS
equation (1.1) with $f\left(u\right)=\lambda \left|u\right|^{\sigma } u$, $\lambda>0$ and $u_{0} \in H^{1} (\R^{n})$
extends globally in time. Genoud--Stuart \cite{GS08} studied (1.1) by using abstract theory developed by Cazenave \cite{C03}. They showed that (1.1) with $f\left(u\right)=\lambda \left|u\right|^{\sigma } u$, $\lambda <0$ and $0<b<\min \left\{2,\;n\right\}$ is well--posed in $H^{1} (\R^{n})$:

$\cdot$  locally if $0<\sigma <\sigma _{1} $,

$\cdot$  globally for any initial data if $0<\sigma <\sigma _{0} $,

$\cdot$  globally for small initial data if $\sigma _{0} \le \sigma <\sigma _{1} $.

\noindent Later, Genoud \cite{G12} and Farah \cite{F16} studied (1.1) with $f\left(u\right)=\lambda \left|u\right|^{\sigma } u$, $\lambda <0$, $0<b<\min \{ 2,\;n\} $ and $\sigma _{0} \le \sigma <\sigma _{1} $ by using sharp Gagliardo--Nirenberg inequalities. They proved that the solution $u$ of (1.1) is globally defined in $H^{1} (\R^{n})$ quantifying the smallness condition of the initial data. Guzm\'{a}n \cite{G17} established the well--posedness for (1.1) with $f\left(u\right)=\lambda \left|u\right|^{\sigma } u$, $\lambda \in \R$ by using Strichartz estimates. Precisely, he showed that:

$\cdot$  if $0<\sigma <\sigma _{0} $, and $0<b<\min \left\{2,\;n\right\}$, then (1.1) is globally well--posed in $L^{2} (\R^{n})$,

$\cdot$  if $0<s\le \min \left\{1,\;\frac{n}{2} \right\}$, $0<b<\tilde{2}$ and $0<\sigma <\sigma _{s} $, then (1.1) is locally well--posed in $H^{s} (\R^{n})$, where
\begin{equation} \label{GrindEQ__1_7_}
\tilde{2}=\left\{\begin{array}{l} {\frac{n}{3} ,\;n=1,\;2,\;3,} \\ {2,\;n\ge 4,} \end{array}\right.
\end{equation}

$\cdot$  if $0<s\le \min \left\{1,\;\frac{n}{2} \right\}$, $0<b<\tilde{2}$ and $\sigma _{0} <\sigma <\sigma _{s} $, then (1.1) is globally well--posed in $H^{s} (\R^{n})$ for small initial data.

Recently, the authors in \cite{AK21} improved the local well--posedness result of \cite{G17} by extending the validity of $s$ and $b$. More precisely, they proved that (1.1) is locally well--posed in $H^{s} \left(\R^{n}\right)$, if $0\le s<\min \left\{n,\;\frac{n}{2} +1\right\}$, $0<b<\hat{2}$ and $0<\sigma <\sigma _{s}$. But they didn't study the global well--posedness of (1.1).

We also refer the reader to [5, 7--10] for the scattering and blow--up results for (1.1).

The purpose of this paper is to obtain the small data global well--posedness and scattering results in $H^{s} (\R^{n})$ with $0<s<\min \left\{n,\;\frac{n}{2} +1\right\}$ for the INLS equation (1.1). To arrive at this goal, we establish various nonlinear estimates and use the contraction mapping principle based on Strichartz estimates. Our results improve the global well--posedness result of Guzm\'{a}n \cite{G17} by extending the validity of $s$ and $b$.

The first main result of this paper is the following small data global well--posedness result in $H^{s} (\R^{n})$ with $0<s<\min \left\{n,\;\frac{n}{2} +1\right\}$.

\begin{thm}\label{thm 1.3.}
Let $n\in \N$, $0<s<\min \left\{n,\;\frac{n}{2} +1\right\}$, $0<b<\hat{2}$ and $\sigma _{0} <\sigma <\sigma _{s} $. Assume that $f$ is of class $X\left(\sigma,s,b\right)$. Then there exists a corresponding $\delta >0$ such that for any $u_{0} \in H^{s} (\R^{n})$ satisfying $\left\| u_{0} \right\| _{H^{s} (\R^{n})} \le \delta $, $(1.1)$ has a unique global solution satisfying
\begin{equation} \label{GrindEQ__1_8_}
u\in C\left(\R,\;H^{s} (\R^{n})\right)\bigcap L^{\gamma \left(p\right)} \left(\R,\;H_{p}^{s} (\R^{n})\right),
\end{equation}
for any admissible pair $\left(\gamma \left(p\right),\;p\right)$.
\end{thm}
If $0<s\le 1$, we can see that
\begin{equation} \label{GrindEQ__1_9_}
\hat{2}=\left\{\begin{array}{l} {2,\;n\ge 4,} \\ {\min \left\{2,\;\frac{5}{2}-s\right\},\;n=3,} \\ {n-s,\;n=1,\;2,} \end{array}\right.
\end{equation}
and $\left\lceil s\right\rceil =1<\sigma +1$. Thus we have the following global well--poseness results in $H^{s} (\R^{n})$ with $0<s\le 1$.

\begin{cor}\label{cor 1.4.}
Let $n\in \N$, $0<s<1$, $0<b<\hat{2}$ and $\sigma _{0} <\sigma <\sigma _{s} $. Assume that $f\left(u\right)=\lambda \left|u\right|^{\sigma } u$ with $\lambda \in \C$. Then there exists a corresponding $\delta >0$ such that for any $u_{0} \in H^{s} (\R^{n})$ satisfying $\left\| u_{0} \right\| _{H^{s} (\R^{n})} \le \delta $, $(1.1)$ has a unique global solution satisfying $(1.8)$.
\end{cor}

\begin{cor}\label{cor 1.5.}
Let $n\ge 2$, $s=1$, $0<b<\hat{2}$ and $\sigma _{0} <\sigma <\sigma _{s} $. Assume that $f\left(u\right)=\lambda \left|u\right|^{\sigma } u$ with $\lambda \in \C$. Then there exists a corresponding $\delta >0$ such that for any $u_{0} \in H^{1} (\R^{n})$ satisfying $\left\| u_{0} \right\| _{H^{1} (\R^{n})} \le \delta $, $(1.1)$ has a unique global solution satisfying $(1.8)$.
\end{cor}

\begin{rem}\label{rem 1.6}
\textnormal{Corollary 1.5 and Corollary 1.6 improve the global well--poseness result of \cite{G17}. First, Guzm\'{a}n \cite{G17} didn't treat the case $\frac{1}{2}<s<1$ for $n=1$, while Corollary 1.5 holds for any $n\in \N$ and $0<s<1$. Next, we can easily see that $\tilde{2}<\hat{2}$ for $n\le 3$ and $0<s\le \min \left\{1,\;\frac{n}{2} \right\}$.}
\end{rem}

We also have the following scattering result.

\begin{thm}\label{thm 1.7.}
Let $n\in \N$, $0<s<\min \left\{n,\;\frac{n}{2} +1\right\}$, $0<b<\hat{2}$ and $\sigma _{0} <\sigma <\sigma _{s} $. Assume that $f$ is of class $X\left(\sigma,s,b\right)$. Let $u_{0} \in H^{s} (\R^{n})$ satisfying $\left\| u_{0} \right\| _{H^{s} (\R^{n})} \le \delta $, where $\delta $ is as in Theorem 1.3, and let $u\left(x,\;t\right)$ be the solution of $(1.1)$. Then there exist $u_{0}^{\pm } \in H^{s} (\R^{n})$ such that
\begin{equation} \label{GrindEQ__1_10_}
{\mathop{\lim }\limits_{t\to \pm \infty }} \left\| u(t)-e^{it\Delta } u_{0}^{\pm } \right\| _{H^{s} (\R^{n})} =0.
\end{equation}
\end{thm}
This paper is organized as follows. In Section 2, we introduce some natation and give some preliminary results related to our problem. In Section 3, we establish various nonlinear estimates. In Section 4, we prove Theorem 1.3 and Theorem 1.4.

\noindent
\section{Preliminaries}

Let us introduce some natation used in this paper. $\mathfrak{F}$ denotes the Fourier transform; $\mathfrak{F}^{-1} $ denotes the inverse Fourier transform. Given $p\in [1,\;\infty ]$, we denote by $p'$ the conjugate exponent of $p$, i.e. $1/p+1/p'=1$. $C\left(>0\right)$ stands for the universal constant, which can be different at different places. $a\lesssim b$ means $a\le Cb$ for some constant $C>0$. For $s\in\R$, we denote by $\left[s\right]$ the largest integer which is less than or equals to $s$ and by $\left\lceil s\right\rceil $ the minimal integer which is larger than or equals to $s$. For $U \subset \R^{n}$, $\chi_{U}$ denotes the characteristic function of $U$, i.e. $\chi_{U}(x)=1$ for $x\in U$, and $\chi_{U}(x)=0$ for $x \in U^{C}$. For the multi-index $\alpha =\left(\alpha _{1} ,\;\alpha _{2} \;,\;\ldots ,\;\alpha _{n} \right)$, denote
\[D^{\alpha } =\partial _{x_{1} }^{\alpha _{1} } \cdots \partial _{x_{n} }^{\alpha _{n} },\; \left|\alpha \right|=\left|\alpha _{1} \right|+\cdots +\left|\alpha _{n} \right|.\]
Given normed spaces $X$ and $Y$, $X\subset Y$ means that $X$ is continuously embedded in $Y$, i.e. there exists a constant $C\left(>0\right)$ such that $\left\| f\right\| _{Y} \le C\left\| f\right\| _{X} $ for all $f\in X$. For a function $f(z)$ defined for a complex variable $z$ and for a positive integer $k$, $k$--th order derivative of $f(z)$ is defined by
\[f^{(k)}(z):=\left(\frac{\partial ^{k} f}{\partial z^{k} } ,\; \frac{\partial ^{k} f}{\partial z^{k-1} \partial \bar{z}} ,\; {\dots},\; \frac{\partial ^{k} f}{\partial z\partial \bar{z}^{k-1} }, \;\frac{\partial ^{k} f}{\partial \bar{z}^{k} } \right),\]
where
\[\frac{\partial f}{\partial z}=\frac{1}{2} \left(\frac{\partial f}{\partial x} -i\frac{\partial f}{\partial y} \right),\; \frac{\partial f}{\partial \bar{z}}=\frac{1}{2} \left(\frac{\partial f}{\partial x} +i\frac{\partial f}{\partial y} \right).\]
We also define its norm as
\[\left|f^{\left(k\right)} (z)\right|:=\sum _{i=0}^{k}\left|\frac{\partial ^{k} f}{\partial z^{k-i} \partial \bar{z}^{i} } \right| .\]
As in \cite{WHHG11}, for $s\in \R$ and $1<p<\infty $, we denote by $H_{p}^{s} (\R^{n} )$ and $\dot{H}_{p}^{s} (\R^{n} )$ the nonhomogeneous Sobolev space and homogeneous Sobolev space, respectively. The norms of these spaces are given as
\[\;\left\| f\right\| _{H_{p}^{s} (\R^{n} )} =\left\| J_{s} f\right\| _{L^{p} (\R^{n} )} , \;\left\| f\right\| _{\dot{H}_{p}^{s} (\R^{n} )} =\left\| I_{s} f\right\| _{L^{p} (\R^{n} )} ,\]
where $J_{s} =\mathfrak{F}^{-1} \left(1+\left|\xi \right|^{2} \right)^{\frac{s}{2} } \mathfrak{F}$ and $I_{s} =\mathfrak{F}^{-1} \left|\xi \right|^{s} \mathfrak{F}$. As usual, we abbreviate $H_{2}^{s} (\R^{n} )$ and $\dot{H}_{2}^{s} (\R^{n} )$ as $H^{s} (\R^{n} )$ and $\dot{H}^{s} (\R^{n} )$, respectively. For $I\subset \R$ and $\gamma \in \left[1,\infty \;\right]$, the norm of $L^{\gamma } \left(I,\;X(\R^{n})\right)$ is defined by
\[\left\| f\right\| _{L^{\gamma } \left(I,\;X(\R^{n})\right)} =\left(\int _{I}\left\| f\right\| _{X(\R^{n})}^{\gamma } dt \right)^{\frac{1}{\gamma } } ,\]
with a usual modification when $\gamma =\infty $, where $X(\R^{n} )$ is a normed space on $\R^{n} $ such as Lebesgue space or Sobolev space. We will omit $\R^{n} $ in various function spaces, if there is no confusion.

Next, we recall some useful facts and estimates which are used in this paper.

First of all, we state the following fundamental result which is useful to treat the case $s\ge 1$.

\begin{lem}[\cite{AK21}]\label{lem 2.1.}
Let $s>0$, $1<p<\infty $ and $v=s-\left[s\right]$. Then $\sum _{\left|\alpha \right|=\left[s\right]}\left\|D^{\alpha }f\right\| _{\dot{H}_{p}^{v}}  $ is an equivalent norm on $\dot{H}_{p}^{s} (\R^{n})$.
\end{lem}
The following lemma is the well-known fractional product rule. See \cite{CW91,T00} for example.
\begin{lem}[Fractional Product Rule]\label{lem 2.2.}
Let $s\ge 0$, $1<r,\;r_{2} ,\;p_{1} <\infty $, $1<r_{1} ,\;p_{2} \le \infty $. Assume that
\[\frac{1}{r} =\frac{1}{r_{i} } +\frac{1}{p_{i} }\;(i=1,\;2).\]
Then we have
\begin{equation} \label{GrindEQ__2_1_}
\left\| fg\right\| _{\dot{H}_{r}^{s} } \lesssim\left\| f\right\| _{r_{1} } \left\| g\right\| _{\dot{H}_{p_{1} }^{s} } +\left\| f\right\| _{\dot{H}_{r_{2} }^{s} } \left\| g\right\| _{p_{2} } .
\end{equation}
\end{lem}
\begin{cor}\label{cor 2.3.}
Let $s\ge 0$, $q\in \N$. Let $1<r,\;r_{k}^{i} <\infty $ for $1\le i,\;k\le q$. Assume that
\[\frac{1}{r} =\sum _{i=1}^{q}\frac{1}{r_{k}^{i} }  , \]
for any $1\le k\le q$. Then we have
\begin{equation} \label{GrindEQ__2_2_}
\left\| \prod _{i=1}^{q}f_{i}  \right\| _{\dot{H}_{r}^{s} } \lesssim\sum _{k=1}^{q}\left(\left\| f_{k} \right\| _{\dot{H}_{r_{k}^{k} }^{s} } \prod _{i\in I_{k} }\left\| f_{i} \right\| _{r_{k}^{i} }  \right) ,
\end{equation}
where $I_{k} =\left\{i\in \N:\;1\le i\le q,\;i\ne k\right\}$.
\end{cor}
\begin{proof} We can easily prove (2.2) by using Lemma 2.2, H\"{o}lder inequality and induction on $q$.
\end{proof}
The following lemma is the well-known fractional chain rule. See \cite{CW91,KPV93} for example.
\begin{lem}[Fractional Chain Rule]\label{lem 2.3.}
Suppose $G\in C^{1} (\C)$ and $s\in \left(0,\;1\right)$. Then, for $1<r,\;r_{2} <\infty $ and $1<r_{1} \le \infty $ satisfying $\frac{1}{r} =\frac{1}{r_{1} } +\frac{1}{r_{2} } $,
\begin{equation} \label{GrindEQ__2_4_}
\left\| G(u)\right\| _{\dot{H}_{r}^{s} } \lesssim\left\| G'(u)\right\| _{r_{1} } \left\|u\right\|_{\dot{H}_{r_{2} }^{s} } .
\end{equation}
\end{lem}

Next, we recall some embeddings on Sobolev spaces. See \cite{WHHG11} for example.
\begin{lem}\label{lem 2.5.}
Let $-\infty <s_{2} \le s_{1} <\infty $ and $1<p_{1} \le p_{2} <\infty $ with $s_{1} -\frac{n}{p_{1} } =s_{2} -\frac{n}{p_{2} } $. Then we have the following embeddings:
\[\dot{H}_{p_{1} }^{s_{1} } \subset \dot{H}_{p_{2} }^{s_{2} } ,\; H_{p_{1} }^{s_{1} } \subset H_{p_{2} }^{s_{2} } .\]
\end{lem}
\begin{lem}\label{lem 2.6.}
Let $-\infty <s<\infty $ and $1<p<\infty $. Then we have
\[H_{p}^{s+\varepsilon } \subset H_{p}^{s}\;(\varepsilon >0),\; H_{p}^{s} =L^{p} \bigcap \dot{H}_{p}^{s}\;(s>0)\]
\end{lem}
We recall the convexity H\"{o}lder inequality in Sobolev spaces. See Proposition 1.21 of \cite{WHHG11} for example.
\begin{lem}[Convexity H\"{o}lder inequality]\label{lem 2.7.}
 Let $1<p,\;p_{i} <\infty $, $0\le \theta _{i} \le 1$, $s_{i} ,\;s\in \R{\rm ,\; (}i=1,\;\ldots ,\;N{\rm )}$, $\sum _{i=1}^{N}\theta _{i}  =1$, $s=\sum _{i=1}^{N}\theta _{i}  s_{i} $, ${1 \mathord{\left/{\vphantom{1 p}}\right.\kern-\nulldelimiterspace} p} =\sum _{i=1}^{N}{\theta _{i}  \mathord{\left/{\vphantom{\theta _{i}  p_{i} }}\right.\kern-\nulldelimiterspace} p_{i} }  $. Then we have $\bigcap _{i=1}^{N}\dot{H}_{p_{i} }^{s_{i} }  \subset \dot{H}_{p}^{s} $ and for any $v\in \bigcap _{i=1}^{N}\dot{H}_{p_{i} }^{s_{i} }  $,
\begin{equation} \label{GrindEQ__2_4_}
\left\| v\right\| _{\dot{H}_{p}^{s} } \le \prod _{i=1}^{N}\left\| v\right\| _{\dot{H}_{p_{i} }^{s_{i} } }^{\theta _{i} }  .
\end{equation}
\end{lem}
\begin{defn}\label{defn 2.8.}
\textnormal{A pair $(\gamma(p),\;p)$ is said to be \textbf{Schr\"{o}dinger admissible} if
\begin{equation} \label{GrindEQ__2_5_}
\left\{\begin{array}{l} {2\le p\le \frac{2n}{n-2} ,\;n\ge 3,} \\ {2\le p<\infty ,\;n=2,} \\ {2\le p\le \infty ,\;n=1,} \end{array}\right.
\end{equation}
and
\begin{equation} \label{GrindEQ__2_6_}
\frac{2}{\gamma (p)} =\frac{n}{2} -\frac{n}{p} .
\end{equation}
For a given interval $I$, we also define the following \textbf{Strichartz norms}:
\begin{equation} \label{GrindEQ__2_7_}
\left\| u\right\| _{S\left(I,\;\dot{H}^{s} \right)}={\mathop{\sup }\limits_{\left(\gamma \left(r\right),\;r\right)\in A}} \left\| u\right\| _{L^{\gamma \left(r\right)} \left(I,\;\dot{H}_{r}^{s} \right)} ,\;\left\| u\right\| _{S\left(I,\;H^{s} \right)}={\mathop{\sup }\limits_{\left(\gamma \left(r\right),\;r\right)\in A}} \left\| u\right\| _{L^{\gamma \left(r\right)} \left(I,\;H_{r}^{s} \right)} ,
\end{equation}
and the \textbf{dual Strichartz norms}:
\begin{equation} \label{GrindEQ__2_8_}
\left\| u\right\| _{S'\left(I,\;\dot{H}^{s} \right)}={\mathop{\inf }\limits_{\left(\gamma \left(r\right),\;r\right)\in A}} \left\| u\right\| _{L^{\gamma \left(r\right)^{{'} } } \left(I,\;\dot{H}_{r'}^{s} \right)} ,\;\left\| u\right\| _{S'\left(I,\;H^{s} \right)}={\mathop{\inf }\limits_{\left(\gamma \left(r\right),\;r\right)\in A}} \left\| u\right\| _{L^{\gamma \left(r\right)^{{'} } } \left(I,\;H_{r'}^{s} \right)} ,
\end{equation}
where $s\in \R$ and $A=\left\{\left(\gamma \left(r\right),\;r\right);\;\left(\gamma \left(r\right),\;r\right)\;\textnormal{is admissible}\right\}$. We shall abbreviate $S\left(\R,\;X\right)$ and $S'\left(\R,\;X\right)$ as $S\left(X\right)$ and $S'\left(X\right)$, respectively, where $X$ is $\dot{H}^{s}$ or $H^{s} $. We write $S\left(L^{2} \right)$ and $S'\left(L^{2} \right)$ instead of $S(\dot{H}^{0})$ and $S'(\dot{H}^{0})$, respectively.}
\end{defn}
We end this section with recalling the well-known Strichartz estimates which are foundational tool to obtain the well--posedness results. See \cite{LP15,WHHG11} for instance.

\begin{lem}[Strichartz estimates]\label{lem 2.9.}
Let $S(t)=e^{it\Delta } $ and $s\in \R$. Then we have
\begin{equation} \label{GrindEQ__2_9_}
\left\| S(t)\phi \right\| _{S(\dot{H}^{s})} \lesssim\left\| \phi \right\| _{\dot{H}^{s} } ,
\end{equation}
\begin{equation} \label{GrindEQ__2_10_}
\left\| \int _{0}^{t}S(t-\tau )f(\tau )d\tau  \right\| _{S(\dot{H}^{s})} \lesssim\left\| f\right\| _{S'(\dot{H}^{s})} .
\end{equation}
\end{lem}

\section{Nonlinear estimates}

First, we establish the following important nonlinear estimates in Sobolev spaces.

\begin{lem}\label{lem 3.1.}
Let $1<p,\;r<\infty $, $0\le \delta \le \min \left\{s_{0} ,\;s\right\}<\infty, \;\sigma>0$ and $\left\lceil s-\delta \right\rceil \le\sigma+1$. Assume that $f\in C^{\left\lceil s-\delta \right\rceil } $ satisfies the following condition:
\begin{equation} \label{GrindEQ__3_1_}
\left|f^{\left(k\right)} \left(z\right)\right|\lesssim\left|z\right|^{\sigma +1-k} ,
\end{equation}
for any $0\le k\le \left\lceil s-\delta \right\rceil $ and $z\in \C$. Suppose that
\begin{equation} \label{GrindEQ__3_2_}
\frac{1}{p} =\sigma \left(\frac{1}{r} -\frac{s_{0} }{n} \right)+\frac{1}{r} -\frac{\delta }{n} ,\; \frac{1}{r} -\frac{s_{0} }{n} >0.
\end{equation}
Then we have
\begin{equation} \label{GrindEQ__3_3_}
\left\| f\left(u\right)\right\| _{\dot{H}_{p}^{s-\delta } } \lesssim\left\| u\right\| _{\dot{H}_{r}^{s_{0} } }^{\sigma} \left\| u\right\| _{\dot{H}_{r}^{s} } .
\end{equation}
\end{lem}

\begin{proof}
 If $s-\delta \in \N$, the proof can be found in Lemma 2.1 of \cite{W04}. Thus it suffices to consider the case $s-\delta \notin \N$. We use the argument similar to that used in the proof of Lemma 2.1 of \cite{W04}.
First, we consider the case $s-\delta <1$. Putting
\[\frac{1}{p_{1} }:=\sigma \left(\frac{1}{r} -\frac{s_{0} }{n} \right), \; \frac{1}{p_{2} } :=\frac{1}{r} -\frac{\delta }{n} ,\]
and using Lemma 2.5, we have $\dot{H}_{r}^{s_{0} } \subset L^{\frac{p_{1}}{\sigma}}$, $\dot{H}_{r}^{s} \subset \dot{H}_{p_{2} }^{s-\delta } $. Hence, using Lemma 2.4 (fractional chain rule), (3.1), (3.2) and H\"{o}lder inequality , we have
\[\left\| f\left(u\right)\right\| _{\dot{H}_{p}^{s-\delta } } \lesssim\left\| f'\left(u\right)\right\| _{p_{1} } \left\| u\right\| _{\dot{H}_{p_{2} }^{s-\delta } } \lesssim\left\| u\right\| _{\frac{p_{1}}{\sigma}}^{\sigma } \left\| u\right\| _{\dot{H}_{p_{2} }^{s-\delta } } \lesssim\left\| u\right\| _{\dot{H}_{r}^{s_{0} } }^{\sigma } \left\| u\right\| _{\dot{H}_{r}^{s} } .\]
Next, we consider the case $s-\delta >1$. By Lemma 2.1, we have
\[\left\| f\left(u\right)\right\| _{\dot{H}_{p}^{s-\delta } } \lesssim\sum _{\left|\alpha \right|=\left[s-\delta \right]}\left\| D^{\alpha } f\left(u\right)\right\| _{\dot{H}_{p}^{v} }  ,\]
where $v=s-\delta -\left[s-\delta \right]$. Without loss of generality and for simplicity, we assume that $f$ is a function of a real variable. It follows from the Leibniz rule of derivatives that
\begin{equation} \label{GrindEQ__3_4_}
D^{\alpha } f\left(u\right)=\sum _{q=1}^{\left|\alpha \right|}\sum _{\Lambda _{\alpha }^{q} }C_{\alpha ,\;q} f^{\left(q\right)} \left(u\right)\prod _{i=1}^{q}D^{\alpha _{i} } u   ,
\end{equation}
where $\Lambda _{\alpha }^{q} =\left(\alpha _{1} +\cdots +\alpha _{q} =\alpha ,\;\left|\alpha _{i} \right|\ge 1\right)$. Hence it suffices to show that
\begin{equation} \label{GrindEQ__3_5_}
A\equiv \left\| f^{\left(q\right)} \left(u\right)\prod _{i=1}^{q}D^{\alpha _{i} } u \right\| _{\dot{H}_{p}^{v} } \lesssim\left\| u\right\| _{\dot{H}_{r}^{s_{0} } }^{\sigma } \left\| u\right\| _{\dot{H}_{r}^{s} } ,
\end{equation}
where $\left[s-\delta \right]\ge q\ge 1$, $\left|\alpha _{1} \right|+\cdots +\left|\alpha _{q} \right|=\left[s-\delta \right]$, $\left|\alpha _{i} \right|\ge 1$ and $v=s-\delta -\left[s-\delta \right]$.

\noindent We divide the proof of (3.5) in two cases: $s\le s_{0} +1$ and $s>s_{0} +1$.

\textbf{Case 1}: We consider the case $s\le s_{0} +1$. Put
\begin{equation} \label{GrindEQ__3_6_}
\frac{1}{a}:=\frac{1}{r} -\frac{s_{0} }{n} ,\; \frac{1}{b}:=\frac{1}{r} -\frac{s_{0} -v}{n} ,\; \frac{1}{\tilde{a}_{q} }:=\frac{1}{r} -\frac{s-\left|\alpha _{q} \right|}{n} ,
\end{equation}
\begin{equation} \label{GrindEQ__3_7_}
\frac{1}{a_{i} }:=\frac{1}{r} -\frac{s_{0} -\left|\alpha _{i} \right|}{n} ,\; \frac{1}{b_{i} } :=\frac{1}{r} -\frac{s-\left|\alpha _{i} \right|-v}{n} ,\; i=1,\;\ldots ,\;q.
\end{equation}
We can see that $v=s-\delta -\left[s-\delta \right]\le s-\delta -1\le s_{0} $, which implies that $b>0$ and $\dot{H}_{r}^{s_{0} } \subset \dot{H}_{b}^{v}$. We can also see that $s-\left|\alpha _{i} \right|-v\ge s-\left[s-\delta \right]-v=\delta \ge 0$ and $s-\left|\alpha _{i} \right|-v\le s-1\le s_{0} $, which imply that $b_{i} >0$ and $\dot{H}_{r}^{s}\subset \dot{H}_{b_{i} }^{\left|\alpha _{i} \right|+v} $. We also have $\dot{H}_{r}^{s} \subset \dot{H}_{\tilde{a}_{q} }^{\left|\alpha _{q} \right|} $, since $s\ge \left|\alpha _{i} \right|$. If $q\ge 2$, then we can verify that $\left|\alpha _{i} \right|\le \left[s-\delta \right]-1\le s-1\le s_{0} $, which implies that $a_{i} >0$ and $\dot{H}_{r}^{s_{0} } \subset \dot{H}_{a_{i} }^{\left|\alpha _{i} \right|}$. For $1\le k\le q$, we have
\begin{equation} \label{GrindEQ__3_8_}
\frac{1}{p} =\frac{\sigma -q}{a} +\frac{1}{b} +\frac{1}{\tilde{a}_{q} } +\sum _{i\in I_{{}_{q} } }\frac{1}{a_{i} }  =\frac{\sigma -q+1}{a} +\frac{1}{b_{k} } +\sum _{i\in I_{k} }\frac{1}{a_{i} }  ,
\end{equation}
where $I_{k} =\left\{i\in \N:\;1\le i\le q,\;i\ne k\right\}$. Using (3.8) and Lemma 2.2 (fractional product rule), we have
\begin{equation} \label{GrindEQ__3_9_}
A\lesssim\left\| f^{\left(q\right)} \left(u\right)\right\| _{\dot{H}_{p_{1} }^{v} } \left\| \prod _{i=1}^{q}D^{\alpha _{i} } u \right\| _{r_{1} } +\left\| f^{\left(q\right)} \left(u\right)\right\| _{p_{2} } \left\| \prod _{i=1}^{q}D^{\alpha _{i} } u \right\| _{\dot{H}_{r_{2} }^{v} } \equiv A_{1} +A_{2} ,
\end{equation}
where
\begin{equation} \label{GrindEQ__3_10_}
\frac{1}{p_{1} }:=\frac{\sigma -q}{a} +\frac{1}{b} ,\; \frac{1}{r_{1} }:=\frac{1}{\tilde{a}_{q} } +\sum _{i\in I_{{}_{q} } }\frac{1}{a_{i} }  ,
\end{equation}
\begin{equation} \label{GrindEQ__3_11_}
\frac{1}{p_{2} }:=\frac{\sigma -q+1}{a} ,\; \frac{1}{r_{2} }:=\frac{1}{b_{k} } +\sum _{i\in I_{k} }\frac{1}{a_{i} }  .
\end{equation}
First, we estimate $A_{1} $. Since $q\le\sigma,$ it follows from Lemma 2.4 (fractional chain rule) and (3.1) that
\begin{eqnarray}\begin{split} \label{GrindEQ__3_12_}
\left\| f^{\left(q\right)} \left(u\right)\right\| _{\dot{H}_{p_{1} }^{v} } &\lesssim \left\| f^{\left(q+1\right)} \left(u\right)\right\| _{\frac{a}{\sigma -q} } \left\| u\right\| _{\dot{H}_{b}^{v} } \lesssim \left\| \left|u\right|^{\sigma -q} \right\| _{\frac{a}{\sigma -q} } \left\| u\right\| _{\dot{H}_{b}^{v} } \\
&=\left\| u\right\| _{a}^{\sigma -q} \left\| u\right\| _{\dot{H}_{b}^{v} } \lesssim \left\| u\right\| _{\dot{H}_{r}^{s_{0} } }^{\sigma -q+1},
\end{split}\end{eqnarray}
where we assume $\frac{a}{0}=\infty$ and the last inequality follows from the embeddings: $\dot{H}_{r}^{s_{0} } \subset L^{a},\;\dot{H}_{r}^{s_{0} } \subset \dot{H}_{b}^{v}$. Using H\"{o}lder inequality, we also have
\begin{equation} \label{GrindEQ__3_13_}
\left\| \prod _{i=1}^{q}D^{\alpha _{i} } u \right\| _{r_{1} } \le \left\| D^{\alpha _{q} } u\right\| _{\tilde{a}_{q} } \prod _{i\in I_{q} }\left\| D^{\alpha _{i} } u\right\| _{a_{i} }  \lesssim\left\| u\right\| _{\dot{H}_{\tilde{a}_{q} }^{\left|\alpha _{q} \right|} } \prod _{i\in I_{q} }\left\| u\right\| _{\dot{H}_{a_{i} }^{\left|\alpha _{i} \right|} }  \lesssim\left\| u\right\| _{\dot{H}_{r}^{s_{0} } }^{q-1} \left\| u\right\| _{\dot{H}_{r}^{s} } .
\end{equation}
In view of (3.12) and (3.13), we have
\begin{equation} \label{GrindEQ__3_14_}
A_{1} =\left\| f^{\left(q\right)} \left(u\right)\right\| _{\dot{H}_{p_{1} }^{v} } \left\| \prod _{i=1}^{q}D^{\alpha _{i} } u \right\| _{r_{1} } \lesssim\left\| u\right\| _{\dot{H}_{r}^{s_{0} } }^{\sigma } \left\| u\right\| _{\dot{H}_{r}^{s} } .
\end{equation}
Next, we estimate $A_{2} $. We can easily see that
\begin{equation} \label{GrindEQ__3_15_}
\left\| f^{\left(q\right)} \left(u\right)\right\| _{p_{2} } \lesssim \left\| \left|u\right|^{\sigma +1-q} \right\| _{p_{2} } =\left\| u\right\| _{a}^{\sigma +1-q} \lesssim \left\| u\right\| _{\dot{H}_{r}^{s_{0} } }^{\sigma +1-q} .
\end{equation}
Using (3.11) and Corollary 2.3, we have
\begin{eqnarray}\begin{split} \label{GrindEQ__3_16_}
\left\| \prod _{i=1}^{q}D^{\alpha _{i} } u \right\| _{\dot{H}_{r_{2} }^{v} } &\lesssim\sum _{k=1}^{q}\left(\left\| D^{\alpha _{k} } u\right\| _{\dot{H}_{b_{k} }^{v} } \prod _{i\in I_{k} }\left\| D^{\alpha _{i} } u_{i} \right\| _{a_{i} }  \right)  \\
 &\lesssim\sum _{k=1}^{q}\left(\left\| u\right\| _{\dot{H}_{b_{k} }^{\left|\alpha _{k} \right|+v} } \prod _{i\in I_{k} }\left\| u\right\| _{\dot{H}_{a_{i} }^{\left|\alpha _{i} \right|} }  \right) \lesssim\left\| u\right\| _{\dot{H}_{r}^{s} } \left\| u\right\| _{\dot{H}_{r}^{s_{0} } }^{q-1} {\rm .} \end{split}
\end{eqnarray}
In view of (3.15) and (3.16), we immediately have
\begin{equation} \label{GrindEQ__3_17_}
A_{2} =\left\| f^{\left(q\right)} \left(u\right)\right\| _{p_{2} } \left\| \prod _{i=1}^{q}D^{\alpha _{i} } u \right\| _{\dot{H}_{r_{2} }^{v} } \lesssim\left\| u\right\| _{\dot{H}_{r}^{s_{0} } }^{\sigma } \left\| u\right\| _{\dot{H}_{r}^{s} } .
\end{equation}
Using (3.9), (3.14) and (3.17), we have (3.5).

\textbf{Case 2}: We consider the case $s>s_{0} +1$. Put $\beta _{0} =v$. In \cite{W04}, it was proved that one can choose $\beta _{i} \;\left(i=1,\;\ldots ,\;q\right)$ satisfying the following conditions:
\begin{equation} \label{GrindEQ__3_18_}
0\le \beta _{i} \le \left|\alpha _{i} \right|,\; s_{0} +\beta _{i} \ge \left|\alpha _{i} \right|,\; i=1,\;\ldots ,\;q,
\end{equation}
\begin{equation} \label{GrindEQ__3_19_}
\sum _{i=0}^{q}\beta _{i}  =s-s_{0} .
\end{equation}
For details, see the proof of Lemma 2.1. Put
\begin{equation} \label{GrindEQ__3_20_}
\frac{1}{c_{i} } :=\frac{1}{r} -\frac{s_{0} +\beta _{i} -\left|\alpha _{i} \right|}{n} ,\; i=1,\;\ldots ,\;q.
\end{equation}
Using (3.18) and (3.19), we can see that
\begin{equation} \label{GrindEQ__3_21_}
c_{i} >0,\; s_{0} \le s_{0} +\beta _{i} \le s,\; i=0,\;\ldots ,\;q.
\end{equation}
It also follows from (3.2) that
\begin{equation} \label{GrindEQ__3_22_}
\frac{1}{p} =\frac{\sigma -q+1}{a} +\sum _{i=1}^{q}\frac{1}{c_{i} }  ,
\end{equation}
where $a$ is given in (3.6). Using Lemma 2.2 (fractional product rule), we have
\begin{equation} \label{GrindEQ__3_23_}
A\lesssim\left\| f^{\left(q\right)} \left(u\right)\right\| _{\dot{H}_{p_{3} }^{v} } \left\| \prod _{i=1}^{q}D^{\alpha _{i} } u \right\| _{r_{3} } +\left\| f^{\left(q\right)} \left(u\right)\right\| _{p_{3} } \left\| \prod _{i=1}^{q}D^{\alpha _{i} } u \right\| _{\dot{H}_{r_{3} }^{v} } \equiv A_{3} +A_{4} ,
\end{equation}
where
\begin{equation} \label{GrindEQ__3_24_}
\frac{1}{p_{3} }:=\frac{\sigma -q+1}{a} ,\; \frac{1}{r_{3} }:=\sum _{i=1}^{q}\frac{1}{c_{i} }.
\end{equation}
First, we estimate $A_{3} $. Since $q\le\sigma,$ it follows from Lemma 2.4 (fractional chain rule) and (3.1) that
\begin{equation} \label{GrindEQ__3_25_}
\left\| f^{\left(q\right)} \left(u\right)\right\| _{\dot{H}_{p_{3} }^{v} } \lesssim\left\| f^{\left(q+1\right)} \left(u\right)\right\| _{\frac{a}{\sigma -q} } \left\| u\right\| _{\dot{H}_{a}^{v} } \lesssim\left\| u\right\| _{\dot{H}_{r}^{s_{0} } }^{\sigma -q} \left\| u\right\| _{\dot{H}_{r}^{s_{0} +\beta _{0} } } ,
\end{equation}
where we assume $\frac{a}{0}=\infty$ and the last inequality follows from the embeddings: $\dot{H}_{r}^{s_{0} +v} \subset \dot{H}_{a}^{v} ,\;\dot{H}_{r}^{s_{0} } \subset L^{a} $. We also have
\begin{equation} \label{GrindEQ__3_26_}
\left\| \prod _{i=1}^{q}D^{\alpha _{i} } u \right\| _{r_{3} } \lesssim\prod _{i=1}^{q}\left\| u\right\| _{\dot{H}_{c_{i} }^{\left|\alpha _{i} \right|} }  \lesssim\prod _{i=1}^{q}\left\| u\right\| _{\dot{H}_{r}^{s_{0} +\beta _{i} } }  ,
\end{equation}
where the last inequality follows from the embedding $\dot{H}_{r}^{s_{0} +\beta _{i} } \subset \dot{H}_{c_{i} }^{\left|\alpha _{i} \right|} $.
For $i=0,\;\ldots ,\;q$, we can take $0\le \theta _{i} \le 1$ satisfying $s_{0} +\beta _{i} =\theta _{i} s_{0} +\left(1-\theta _{i} \right)s$, since $s_{0} \le s_{0} +\beta _{i} \le s$.
Using (3.19), we can easily see that $\sum _{i=0}^{q}\theta _{i}  =q$ and $\sum _{i=0}^{q}\left(1-\theta _{i} \right) =1$. Hence, it follows from Lemma 2.7 (convexity H\"{o}lder inequality) that
\begin{equation} \label{GrindEQ__3_27_}
\prod _{i=0}^{q}\left\| u\right\| _{\dot{H}_{r}^{s_{0} +\beta _{i} } }  \lesssim\left\| u\right\| _{\dot{H}_{r}^{s_{0} } }^{q} \left\| u\right\| _{\dot{H}_{r}^{s} } .
\end{equation}
In view of (3.25)--(3.27), we have
\begin{equation} \label{GrindEQ__3_28_}
A_{3} \lesssim\left\| u\right\| _{\dot{H}_{r}^{s_{0} } }^{\sigma } \left\| u\right\| _{\dot{H}_{r}^{s} } .
\end{equation}
Next, we estimate $A_{4} $. Using H\"{o}lder inequality, we immediately have
\begin{equation} \label{GrindEQ__3_29_}
\left\| f^{\left(q\right)} \left(u\right)\right\| _{p_{3} } \lesssim\left\| \left|u\right|^{\sigma +1-q} \right\| _{p_{3} }=\left\| u\right\| _{a}^{\sigma +1-q} \lesssim\left\| u\right\| _{\dot{H}_{r}^{s_{0} } }^{\sigma +1-q} .
\end{equation}
Using Corollary 2.3, (3.20) and (3.24), we have
\begin{eqnarray}\begin{split} \label{GrindEQ__3_30_}
\left\| \prod _{i=1}^{q}D^{\alpha _{i} } u \right\| _{\dot{H}_{r_{3} }^{v} } &\lesssim\sum _{k=1}^{q}\left(\left\| D^{\alpha _{k} } u\right\| _{\dot{H}_{{}_{c_{k} } }^{v} } \prod _{i\in I_{k} }\left\| D^{\alpha _{i} } u_{i} \right\| _{c_{i} }  \right) \\
&\lesssim\sum _{k=1}^{q}\left(\left\| u\right\| _{\dot{H}_{c_{k} }^{\left|\alpha _{k} \right|+v} } \prod _{i\in I_{k} }\left\| u\right\| _{\dot{H}_{c_{i} }^{\left|\alpha _{i} \right|} }  \right)\\
&\lesssim\sum _{k=1}^{q}\left(\left\| u\right\| _{\dot{H}_{r}^{s_{0} +\beta _{k} +\beta _{0} } } \prod _{i\in I_{k} }\left\| u\right\| _{\dot{H}_{r}^{s_{0} +\beta _{i} } }  \right) {\rm ,} \end{split}
\end{eqnarray}
where the last inequality follows from the embeddings: $\dot{H}_{r}^{s_{0} +\beta _{k} +v} \subset \dot{H}_{c_{k} }^{\left|\alpha _{k} \right|+v} $ and $\dot{H}_{r}^{s_{0} +\beta _{i} } \subset \dot{H}_{c_{i} }^{\left|\alpha _{i} \right|} $.
In view of (3.19), we can also see that $s_{0} +\beta _{k} +\beta _{0} \le s$ for $k=1,\;\ldots ,\;q$. Thus we can take $0\le \tilde{\theta }_{k} \le 1$ satisfying $s_{0} +\beta _{k} +\beta _{0} =\tilde{\theta }_{k} s_{0} +(1-\tilde{\theta }_{k})s$, for $k=1,\;\ldots ,\;q$. In view of (3.19), we can see that
\begin{equation} \nonumber
\tilde{\theta }_{k} +\sum _{i\in I_{k} }\theta _{i}  =q-1,\; (1-\tilde{\theta }_{k} )+\sum _{i\in I_{k} }\left(1-\theta _{i} \right) =1.
\end{equation}
Hence, using Lemma 2.7 (convexity H\"{o}lder inequality), we have
\begin{equation} \label{GrindEQ__3_31_}
\left\| u\right\| _{\dot{H}_{r}^{s_{0} +\beta _{k} +\beta _{0} } } \prod _{i\in I_{k} }\left\| u\right\| _{\dot{H}_{r}^{s_{0} +\beta _{i} } }  \lesssim\left\| u\right\| _{\dot{H}_{r}^{s_{0} } }^{q-1} \left\| u\right\| _{\dot{H}_{r}^{s} } .
\end{equation}
(3.29), (3.30) and (3.31) yield that
\begin{equation} \label{GrindEQ__3_32_}
A_{4} \lesssim\left\| u\right\| _{\dot{H}_{r}^{s_{0} } }^{\sigma } \left\| u\right\| _{\dot{H}_{r}^{s} } .
\end{equation}
Using (3.23), (3.28) and (3.32), we immediately have (3.5). This completes the proof.
\end{proof}

\begin{rem}\label{rem 3.2.}
\textnormal{Lemma 3.1 generalizes Lemma 3.3 of \cite{AK21} and Lemma 2.1 of \cite{W04}. In fact, Lemma 3.3 of \cite{AK21} follows directly from Lemma 3.1 by putting $\delta=0$, $s_{0}=s$. Lemma 2.1 of \cite{W04} also follows directly from Lemma 3.1 by using the embeddings $\dot{B}_{r,2}^{s} \subset \dot{H}_{r}^{s}$, $\dot{H}_{r'}^{s}\subset \dot{B}_{r',2}^{s}$, where $\dot{B}_{p,q}^{s} $ is homogeneous Besov space (see e.g. \cite{WHHG11}).}
\end{rem}

\begin{rem}\label{rem 3.3.}
\textnormal{If $f\left(z\right)$ is a polynomial in $z$ and $\bar{z}$ satisfying $1<\deg \left(f\right)=1+\sigma $, we can see that the assumption $\left\lceil s\right\rceil \le\sigma +1$ can be removed in Lemma 3.2.}
\end{rem}

\begin{rem}[\cite{G17}]\label{rem 3.4.}
\textnormal{Let $B=B\left(0,\;1\right)=\left\{x\in \R^{n} ;\;\left|x\right|\le 1\right\}$ and $b>{\rm 0}$.
If $\frac{n}{\gamma } >b$, then $\left|x\right|^{-b} \in L^{\gamma } \left(B\right)$. And $\left|x\right|^{-b} \in L^{\gamma } \left(B^{C} \right)$, if $\frac{n}{\gamma } <b$.}
\end{rem}

Using Lemma 3.1 and Remark 3.4, we establish the estimates of nonlinearity $|x|^{-b} f\left(u\right).$

We divide our study in two cases: $n\ge 3$ and $n=1,\;2$.

\begin{lem}\label{lem 3.5.}
Let $n\ge 3$, $0<s<\frac{n}{2} +1$, $0<b<\min \left\{2,{\rm \; 1}+\frac{n-2s}{2} \right\}$ and $\sigma _{0} <\sigma <\sigma _{s} $. Assume that $f$ is of class $X\left(\sigma,s,b\right)$. Then we have
\begin{equation} \label{GrindEQ__3_33_}
\left\| |x|^{-b} \left|u\right|^{\sigma } v\right\| _{S'\left(L^{2} \right)} \lesssim\left\| u\right\| _{S\left(H^{s} \right)}^{\sigma } \left\| v\right\| _{S\left(L^{2} \right)} ,
\end{equation}
\begin{equation} \label{GrindEQ__3_34_}
\left\| |x|^{-b} f\left(u\right)\right\| _{S'(\dot{H}^{s})} \lesssim\left\| u\right\| _{S\left(H^{s} \right)}^{\sigma +1}.
\end{equation}
\end{lem}
\begin{proof}
Putting
\begin{equation} \label{GrindEQ__3_35_}
\bar{r}:=\frac{2n}{n-2} ,\; \frac{1}{r_{1} } :=\frac{1}{2} -\frac{1}{n\left(\sigma +1\right)} ,
\end{equation}
we can easily see that $\left(\gamma \left(\bar{r}\right),\bar{r}\right)$ and $\left(\gamma \left(r_{1} \right),r_{1} \right)$ are admissible. Furthermore, we can see that
\begin{equation} \label{GrindEQ__3_36_}
\frac{1}{\gamma \left(\bar{r}\right)^{{'} } } =\frac{\sigma +1}{\gamma \left(r_{1} \right)} .
\end{equation}
For $B=B\left(0,\;1\right)=\left\{x\in \R^{n} ;\;\left|x\right|\le 1\right\}$, we have
\begin{equation} \label{GrindEQ__3_37_}
\left\||x|^{-b} f\left(u\right)\right\| _{S'(\dot{H}^{s})}\le\left\| |x|^{-b} f\left(u\right)\right\| _{L^{\gamma \left(\bar{r}\right)^{{'} } } \left(\R,\;\dot{H}_{\bar{r}'}^{s}\right)} \le C_{1} +C_{2} ,
\end{equation}
\begin{equation} \label{GrindEQ__3_38_}
\left\|  |x|^{-b} \left|u\right|^{\sigma } v\right\| _{S'\left(L^{2} \right)}\le\left\||x|^{-b} f\left(u\right)\right\| _{L^{\gamma \left(\bar{r}\right)^{{'} } } \left(\R,\;L^{\bar{r}'} \right)}\le D_{1} +D_{2} ,
\end{equation}
where
\begin{equation}\label{GrindEQ__3_39_}
C_{1} = \left\|\chi_{B^{C}}|x|^{-b} f\left(u\right)\right\| _{L^{\gamma \left(\bar{r}\right)^{{'} } } \left(\R,\;\dot{H}_{\bar{r}'}^{s}\right)} ,~
C_{2} =\left\| \chi_{B}|x|^{-b} f\left(u\right)\right\| _{L^{\gamma \left(\bar{r}\right)^{{'} } } \left(\R,\;\dot{H}_{\bar{r}'}^{s}\right)},
\end{equation}
\begin{equation}\label{GrindEQ__3_40_}
D_{1} = \left\| \chi_{B^{C}}|x|^{-b} \left|u\right|^{\sigma } v\right\| _{L^{\gamma \left(\bar{r}\right)^{{'} } } \left(\R,\;L^{\bar{r}'} \right)},~
D_{2} = \left\| \chi_{B}|x|^{-b} \left|u\right|^{\sigma } v\right\| _{L^{\gamma \left(\bar{r}\right)^{{'} } } \left(\R,\;L^{\bar{r}'} \right)}.
\end{equation}

First, we estimate $C_{1} $ and $D_{1} $. We can take $s_{0} \left(>0\right)$ satisfying the following system:
\begin{equation} \label{GrindEQ__3_41_}
\left\{\begin{array}{l} {0<\frac{1}{\bar{r}'} -\sigma \left(\frac{1}{r_{1} } -\frac{s_{0} }{n} \right)-\frac{1}{r_{1} } <\frac{b}{n} ,\;} \\ {0<s_{0} <s,} \\ {\frac{1}{r_{1} } >\frac{s_{0} }{n} .} \end{array}\right.
\end{equation}
In fact, we can see that the first equation in (3.41) is equivalent to
\begin{equation} \label{GrindEQ__3_42_}
\frac{n}{2} -\frac{2}{\sigma } <s_{0} <\frac{n}{2} -\frac{2-b}{\sigma } ,
\end{equation}
and the third equation in (3.41) is equivalent to
\begin{equation} \label{GrindEQ__3_43_}
s_{0} <\frac{n}{2} -\frac{1}{\sigma +1} .
\end{equation}
Thus the system (3.41) is equivalent to
\begin{equation} \label{GrindEQ__3_44_}
\max \left\{0,\;\frac{n}{2} -\frac{2}{\sigma } \right\}<s_{0} <\min \left\{\frac{n}{2} -\frac{2-b}{\sigma } ,\;s,\;\frac{n}{2} -\frac{1}{\sigma +1} \right\}.
\end{equation}
It is obvious that $\frac{n}{2} -\frac{1}{\sigma +1} >0$, since $n\ge 3$. One can easily verify that $\frac{n}{2} -\frac{2-b}{\sigma } >0$ if, and only if, $\sigma >\frac{4-2b}{n} $. It is also obvious that
\begin{equation} \label{GrindEQ__3_45_}
\frac{n}{2} -\frac{2}{\sigma } <\min \left\{\frac{n}{2} -\frac{2-b}{\sigma } ,\;\frac{n}{2} -\frac{1}{\sigma +1} \right\}.
\end{equation}
We can see that $\frac{n}{2} -\frac{2}{\sigma } <s$ is equivalent to $\left(n-2s\right)\sigma <4$. Since $\sigma <\sigma _{s} $, we have $\frac{n}{2} -\frac{2}{\sigma } <s$. Hence we can take $s_{0} \left(>0\right)$ satisfying (3.44). Putting
\begin{equation} \label{GrindEQ__3_46_}
\frac{1}{\gamma _{1} } :=\frac{1}{\bar{r}'} -\sigma \left(\frac{1}{r_{1} } -\frac{s_{0} }{n} \right)-\frac{1}{r_{1} } ,
\end{equation}
it follows from (3.41) that $0<\frac{1}{\gamma _{1} } <\frac{b}{n} $. Lemma 2.2 (fractional product rule) yields
\begin{equation} \label{GrindEQ__3_47_}
\left\| \chi_{B^{C}}|x|^{-b} f\left(u\right)\right\| _{\dot{H}_{\bar{r}'}^{s}} \lesssim\left\| \chi_{B^{C}}|x|^{-b} \right\| _{\gamma _{1}} \left\| f\left(u\right)\right\| _{\dot{H}_{p_{1} }^{s} } +\left\| \chi_{B^{C}}|x|^{-b}\right\| _{\dot{H}_{\gamma _{2} }^{s} } \left\| f\left(u\right)\right\| _{p_{2}} ,
\end{equation}
where
\begin{equation} \label{GrindEQ__3_48_}
\frac{1}{p_{1} }:=\sigma \left(\frac{1}{r_{1} } -\frac{s_{0} }{n} \right)+\frac{1}{r_{1} } ,\; \frac{1}{\gamma _{2} }:=\frac{1}{\gamma _{1} } +\frac{s_{0} }{n} ,\; \frac{1}{p_{2} }:=\left(\sigma +1\right)\left(\frac{1}{r_{1} } -\frac{s_{0} }{n} \right).
\end{equation}
It follows from Remark 3.4 and the fact $\frac{n}{\gamma _{1} } <b$ that
$\left\| \chi_{B^{C}}|x|^{-b} \right\| _{L^{\gamma _{1} }}<\infty$. Using the fact $\frac{n}{\gamma _{2} } <b+s_{0} <b+s$, we also have
\begin{equation} \label{GrindEQ__3_49_}
\left\| \chi_{B^{C}}|x|^{-b} \right\| _{\dot{H}_{\gamma _{2} }^{s}}<\infty.
\end{equation}
In fact, putting $\frac{n}{\bar{\gamma}_{2}}:=\left\lceil s\right\rceil-s+\frac{n}{\gamma_{2}}$, we can see that $\frac{n}{\bar{\gamma}_{2}}<b+\left\lceil s\right\rceil$ and $\dot{H}_{\bar{\gamma}_{2} }^{\left\lceil s\right\rceil} \subset \dot{H}_{\gamma_{2} }^{s}$. Hence, it follows from Remark 3.4 that
\begin{eqnarray}\begin{split}\nonumber
\left\|\chi_{B^{C}}|x|^{-b}\right\|_{\dot{H}_{\gamma_{2} }^{s}}&\lesssim\left\|\chi_{B^{C}}|x|^{-b}\right\|_{\dot{H}_{\bar{\gamma}_{2} }^{\left\lceil s\right\rceil}}\lesssim \sum _{\left|\alpha \right|=\left\lceil s\right\rceil}\left\| D^{\alpha }(\chi_{B^{C}}|x|^{-b})\right\| _{\bar{\gamma}_{2}}\\
&\lesssim \left\| \chi_{B^{C}}|x|^{-b-\left\lceil s\right\rceil}\right\| _{\bar{\gamma}_{2}}<\infty.
\end{split}\end{eqnarray}
 We also have
\begin{equation} \label{GrindEQ__3_50_}
\left\| f\left(u\right)\right\| _{p_{2} } \lesssim\left\| \left|u\right|^{\sigma +1} \right\| _{p_{2} } \lesssim\left\| u\right\| _{\dot{H}_{r_{1}}^{s_{0} } }^{\sigma +1} .
\end{equation}
Hence using (3.47), (3.49), (3.50), Lemma 3.1 and Lemma 2.6, we have
\begin{eqnarray}\begin{split} \label{GrindEQ__3_51_}
\left\|\chi_{B^{C}} |x|^{-b} f\left(u\right)\right\| _{\dot{H}_{\bar{r}'}^{s} } &\lesssim\left\| f\left(u\right)\right\| _{\dot{H}_{p_{1} }^{s} } +\left\| f\left(u\right)\right\| _{p_{2} } \lesssim\left\| u\right\| _{\dot{H}_{r_{1} }^{s_{0} } }^{\sigma } \left(\left\| u\right\| _{\dot{H}_{r_{1} }^{s} } +\left\| u\right\| _{\dot{H}_{r_{1} }^{s_{0} } } \right) \\
&\lesssim\left\| u\right\| _{H_{r_{1} }^{s_{0} } }^{\sigma } \left(\left\| u\right\| _{H_{r_{1} }^{s} } +\left\| u\right\| _{H_{r_{1} }^{s_{0} } } \right)\lesssim \left\| u\right\| _{H_{r_{1} }^{s} }^{\sigma +1} .
\end{split}
\end{eqnarray}
It also follows from H\"{o}lder inequality and (3.48) that
\begin{equation} \label{GrindEQ__3_52_}
\left\| \chi_{B^{C}}|x|^{-b} \left|u\right|^{\sigma } v\right\| _{L^{\bar{r}'} } \le \left\| \chi_{B^{C}}|x|^{-b} \right\| _{L^{\gamma _{1} } } \left\| \left|u\right|^{\sigma } v\right\| _{p_{1} } \lesssim\left\| u\right\| _{\dot{H}_{r_{1} }^{s_{0} } }^{\sigma } \left\| v\right\| _{r_{1} } \lesssim\left\| u\right\| _{H_{r_{1} }^{s} }^{\sigma } \left\| v\right\| _{r_{1} } .
\end{equation}
Using (3.51), (3.52), (3.36) and H\"{o}lder inequality, we have
\begin{equation} \label{GrindEQ__3_53_}
C_{1}=\left\| \chi_{B^{C}}|x|^{-b} f\left(u\right)\right\| _{L^{\gamma \left(\bar{r}\right)^{{'} } } \left(\R,\;\dot{H}_{\bar{r}'}^{s} \right)} \lesssim\left\| u\right\| _{L^{\gamma \left(r_{1} \right)} \left(\R,\;H_{r_{1} }^{s} \right)}^{\sigma +1} \le \left\| u\right\| _{S\left(H^{s} \right)}^{\sigma +1} ,
\end{equation}
\begin{eqnarray}\begin{split} \label{GrindEQ__3_54_}
D_{1}=\left\| \chi_{B^{C}}|x|^{-b} \left|u\right|^{\sigma } v\right\| _{L^{\gamma \left(\bar{r}\right)^{{'} } } \left(\R,\;L^{\bar{r}'} \right)} &\lesssim\left\| u\right\| _{L^{\gamma \left(r_{1} \right)} \left(\R,\;H_{r_{1} }^{s} \right)}^{\sigma } \left\| v\right\| _{L^{\gamma \left(r_{1} \right)} \left(\R,\;L^{r_{1} } \right)}\\
&\le\left\| u\right\| _{S\left(H^{s} \right)}^{\sigma } \left\| v\right\| _{S\left(L^{2} \right)} .
\end{split}\end{eqnarray}
Next, we estimate $C_{2} $ and $D_{2} $. We can take $s_{1} \left(>0\right)$ satisfying the following system:
\begin{equation} \label{GrindEQ__3_55_}
\left\{\begin{array}{l} {\frac{1}{\bar{r}'} -\left(\sigma +1\right)\left(\frac{1}{r_{1} } -\frac{s_{1} }{n} \right)>\frac{b+s}{n} ,\;} \\ {0<s_{1} <s,} \\ {\frac{1}{r_{1} } >\frac{s_{1} }{n} .} \end{array}\right.
\end{equation}
In fact, we can see that the first equation in (3.55) is equivalent to
\begin{equation} \label{GrindEQ__3_56_}
\frac{n}{2} -\frac{n-2s+4-2b}{2\left(\sigma +1\right)} <s_{1} .
\end{equation}
Hence the system (3.55) is equivalent to
\begin{equation} \label{GrindEQ__3_57_}
\max \left\{0,\;\frac{n}{2} -\frac{n-2s+4-2b}{2\left(\sigma +1\right)} \right\}<s_{1} <\min \left\{s,\;\frac{n}{2} -\frac{1}{\sigma +1} \right\}
\end{equation}
We can see that $\frac{n}{2} -\frac{n-2s+4-2b}{2\left(\sigma +1\right)} <s$ is equivalent to
\begin{equation} \label{GrindEQ__3_58_}
\sigma \left(n-2s\right)<4-2b.
\end{equation}
If $s\ge \frac{n}{2} $, (3.58) holds for any $\sigma >0$, since $b<2$. If $s<\frac{n}{2} $, (3.58) is equivalent to $\sigma <\frac{4-2b}{n-2s} $.

\noindent We can also see that $\frac{n}{2} -\frac{n-2s+4-2b}{2\left(\sigma +1\right)} <\frac{n}{2} -\frac{1}{\sigma +1} $ is equivalent to $b<1+\frac{n-2s}{2} $. Hence, using the hypothesis of this lemma, we can take $s_{1} \left(>0\right)$ satisfying (3.57). Putting
\begin{equation} \label{GrindEQ__3_59_}
\frac{1}{p_{3} }:=\left(\sigma +1\right)\left(\frac{1}{r_{1} } -\frac{s_{1} }{n} \right),\; \frac{1}{\gamma _{3} }:=\frac{1}{\bar{r}'} -\frac{1}{p_{3} } ,
\end{equation}
\begin{equation} \label{GrindEQ__3_60_}
\frac{1}{\gamma _{4} }:=\frac{1}{\gamma _{3} } -\frac{s_{1} }{n} ,\; \frac{1}{p_{4} } :=\frac{1}{p_{3} } +\frac{s_{1} }{n} =\sigma \left(\frac{1}{r_{1} } -\frac{s_{1} }{n} \right)+\frac{1}{r_{1} } ,
\end{equation}
we have
\begin{equation} \label{GrindEQ__3_61_}
\frac{1}{\bar{r}'} =\frac{1}{p_{3} } +\frac{1}{\gamma _{3} } =\frac{1}{\gamma _{4} } +\frac{1}{p_{4} } .
\end{equation}
We can also see that $\frac{n}{\gamma _{3} } >b+s$ and $\frac{n}{\gamma _{4} } =\frac{n}{\gamma _{3} } -s_{1} >b+s-s_{1} >b$. Thus using Remark 3.4 and the same argument as in the estimate of (3.49), we can see that $\left\| \chi_{B}|x|^{-b} \right\| _{L^{\gamma _{4} }} $ and $\left\|\chi_{B} |x|^{-b} \right\| _{\dot{H}_{\gamma _{3} }^{s} } $ are finite.
Repeating the same argument as in the proof of (3.51) and (3.52), we have
\begin{equation} \label{GrindEQ__3_62_}
\left\| \chi_{B}|x|^{-b} f\left(u\right)\right\| _{\dot{H}_{\bar{r}}^{s} } \lesssim\left\| \chi_{B}|x|^{-b} \right\| _{L^{\gamma _{4} } } \left\| f\left(u\right)\right\| _{\dot{H}_{p_{4} }^{s} } +\left\| \chi_{B}|x|^{-b} \right\| _{\dot{H}_{\gamma _{3} }^{s} } \left\| f\left(u\right)\right\| _{L^{p_{3} } } \lesssim\left\| u\right\| _{H_{r_{1} }^{s} }^{\sigma +1} .
\end{equation}
\begin{equation} \label{GrindEQ__3_63_}
\left\| \chi_{B}|x|^{-b} \left|u\right|^{\sigma } v\right\| _{L^{\bar{r}'} } \le \left\| \chi_{B}|x|^{-b} \right\| _{L^{\gamma _{4} }} \left\| \left|u\right|^{\sigma } v\right\| _{p_{4} } \lesssim\left\| u\right\| _{\dot{H}_{r_{1} }^{s_{1} } }^{\sigma } \left\| v\right\| _{r_{1} } \lesssim\left\| u\right\| _{H_{r_{1} }^{s} }^{\sigma } \left\| v\right\| _{r_{1} } .
\end{equation}
Using (3.62), (3.63), (3.36) and H\"{o}lder inequality, we have
\begin{equation} \label{GrindEQ__3_64_}
C_{2} = \left\|\chi_{B}|x|^{-b} f\left(u\right)\right\| _{L^{\gamma \left(\bar{r}\right)^{{'} } } \left(\R,\;\dot{H}_{\bar{r}'}^{s} \right)} \lesssim\left\| u\right\| _{L^{\gamma \left(r_{1} \right)} \left(\R,\;H_{r_{1} }^{s} \right)}^{\sigma +1} \lesssim\left\| u\right\| _{S\left(H^{s} \right)}^{\sigma +1} ,
\end{equation}
\begin{equation} \label{GrindEQ__3_65_}
D_{2} = \left\| \chi_{B}|x|^{-b} \left|u\right|^{\sigma } v\right\| _{L^{\gamma \left(\bar{r}\right)^{{'} } } \left(\R,\;L^{\bar{r}'} \right)} \lesssim\left\| u\right\| _{S\left(H^{s} \right)}^{\sigma } \left\| v\right\| _{S\left(L^{2} \right)} .
\end{equation}
In virtue of (3.37), (3.38), (3.53), (3.54), (3.64) and (3.65), we get the desired results.
\end{proof}

\begin{lem}\label{lem 3.6.}
Let $n=1,\;2$, $0<s<n$, $0<b<n-s$ and $\sigma _{0} <\sigma <\sigma _{s} $. Assume that $f$ is of class $X\left(\sigma,s,b\right)$. Then we have
\begin{equation} \label{GrindEQ__3_66_}
\left\| |x|^{-b} \left|u\right|^{\sigma } v\right\| _{S'\left(L^{2} \right)} \lesssim\left\| u\right\| _{S\left(H^{s} \right)}^{\sigma } \left\| v\right\| _{S\left(L^{2} \right)} ,
\end{equation}
\begin{equation} \label{GrindEQ__3_67_}
\left\| |x|^{-b} f\left(u\right)\right\| _{S'(\dot{H}^{s})} \lesssim\left\| u\right\| _{S\left(H^{s} \right)}^{\sigma +1} .
\end{equation}
\end{lem}

\begin{proof}
We use the same argument as in the proof of Lemma 3.5 and we only sketch the proof. Since $b+s<n$, there exists $\hat{r}$ large enough such that
\begin{equation} \label{GrindEQ__3_68_}
\frac{b+s}{n} +\frac{1}{\hat{r}} <1,\; 2<\hat{r}<\infty .
\end{equation}
Obviously $\left(\gamma \left(\hat{r}\right),\;\hat{r}\right)$ is admissible. Put
\begin{equation} \label{GrindEQ__3_69_}
\frac{1}{r_{2} } :=\frac{1}{2} -\frac{2}{n\left(\sigma +1\right)} \left(1-\frac{n}{4} +\frac{n}{2\hat{r}} \right).
\end{equation}
One can verify that $\frac{1}{r_{2} } >0$ if, and only if,
\begin{equation} \label{GrindEQ__3_70_}
\sigma >\frac{4-2n}{n} +\frac{2}{\hat{r}} .
\end{equation}
In view of (3.68), we have $\frac{1}{\hat{r}} <\frac{n-b-s}{n} <\frac{n-b}{n} $. Hence we have
\begin{equation} \label{GrindEQ__3_71_}
\frac{4-2b}{n} >\frac{4-2n}{n} +\frac{2}{\hat{r}} .
\end{equation}
Using (3.71) and the hypothesis $\sigma >\frac{4-2b}{n} $, we have (3.70) which is equivalent to $\frac{1}{r_{2} } >0$. Noticing that
\begin{equation} \label{GrindEQ__3_72_}
1-\frac{n}{4} +\frac{n}{2\hat{r}} >1-\frac{n}{4} >0,
\end{equation}
we also have $\frac{1}{r_{2} } <\frac{1}{2} $. Thus we can see that $r_{2} >2$, i.e. $\left(\gamma \left(r_{2} \right),\;r_{2} \right)$ is admissible.
Moreover, we have
\begin{equation} \nonumber
\frac{1}{\gamma \left(\hat{r}\right)^{{'} } } =\frac{\sigma +1}{\gamma \left(r_{2} \right)} .
\end{equation}
For $B=B\left(0,\;1\right)=\left\{x\in \R^{n} ;\;\left|x\right|\le 1\right\}$, we have
\begin{equation} \nonumber
\left\||x|^{-b} f\left(u\right) \right\| _{S'(\dot{H}^{s})}\le\left\| |x|^{-b} f\left(u\right)\right\| _{L^{\gamma \left(\hat{r}\right)^{{'} } } \left(\R,\;\dot{H}_{\hat{r}'}^{s}\right)} \le E_{1} +E_{2} ,
\end{equation}
\begin{equation} \nonumber
\left\||x|^{-b} \left|u\right|^{\sigma } v \right\| _{S'\left(L^{2} \right)}\le\left\||x|^{-b} \left|u\right|^{\sigma } v\right\| _{L^{\gamma \left(\hat{r}\right)^{{'} } } \left(\R,\;L^{\hat{r}'} \right)}\le F_{1} +F_{2} ,
\end{equation}
where
\begin{equation}\nonumber
E_{1} = \left\|\chi_{B^{C}}|x|^{-b} f\left(u\right)\right\| _{L^{\gamma \left(\hat{r}\right)^{{'} } } \left(\R,\;\dot{H}_{\hat{r}'}^{s}\right)} ,~
E_{2} =\left\| \chi_{B}|x|^{-b} f\left(u\right)\right\| _{L^{\gamma \left(\hat{r}\right)^{{'} } } \left(\R,\;\dot{H}_{\hat{r}'}^{s}\right)},
\end{equation}
\begin{equation}\label{GrindEQ__3_40_}
F_{1} = \left\| \chi_{B^{C}}|x|^{-b} \left|u\right|^{\sigma } v\right\| _{L^{\gamma \left(\hat{r}\right)^{{'} } } \left(\R,\;L^{\hat{r}'} \right)},~
F_{2} = \left\| \chi_{B}|x|^{-b} \left|u\right|^{\sigma } v\right\| _{L^{\gamma \left(\hat{r}\right)^{{'} } } \left(\R,\;L^{\hat{r}'} \right)}.
\end{equation}

\noindent First, we estimate $E_{1} $ and $F_{1} $.
We can take $s_{2} \left(>0\right)$ satisfying the following system:
\begin{equation} \label{GrindEQ__3_74_}
\left\{\begin{array}{l} {0<\frac{1}{\hat{r}'} -\sigma \left(\frac{1}{r_{2} } -\frac{s_{2} }{n} \right)-\frac{1}{r_{2} } <\frac{b}{n} ,\;} \\ {0<s_{2} <s,} \\ {\frac{1}{r_{2} } >\frac{s_{2} }{n} .} \end{array}\right.
\end{equation}
In fact, one can easily see that the first equation in (3.74) is equivalent to
\begin{equation} \label{GrindEQ__3_75_}
\frac{n}{2} -\frac{2}{\sigma } <s_{2} <\frac{n}{2} -\frac{2-b}{\sigma } .
\end{equation}
Thus the system (3.74) is equivalent to
\begin{equation} \label{GrindEQ__3_76_}
\max \left\{\frac{n}{2} -\frac{2}{\sigma } ,\;0\right\}<s_{2} <\min \left\{\frac{n}{2} -\frac{2-b}{\sigma } ,\;s,\;\frac{n}{r_{2} } \right\}.
\end{equation}
One can easily verify that $\frac{n}{2} -\frac{2-b}{\sigma } >0$, since $\sigma >\frac{4-2b}{n} $. We can also see that $\frac{n}{2} -\frac{2}{\sigma } <s$, since $\sigma <\sigma _{s} $. On the other hand, $\frac{n}{2} -\frac{2}{\sigma } <\frac{n}{r_{2} } $ is equivalent to $\frac{n}{2} +\frac{2}{\sigma } >\frac{n}{\hat{r}} $. It is trivial that $\frac{n}{2} +\frac{2}{\sigma } >\frac{n}{\hat{r}} $, since $\hat{r}>2$. Hence, we can take $s_{2} \left(>0\right)$ satisfying (3.76). Repeating the same argument as in the estimates of $C_{1} $ and $D_{1} $ in Lemma 3.5, we can get
\begin{equation} \label{GrindEQ__3_77_}
E_{1} \lesssim\left\| u\right\| _{S\left(H^{s} \right)}^{\sigma +1} ,~F_{1} \lesssim\left\| u\right\| _{S\left(H^{s} \right)}^{\sigma } \left\| v\right\| _{S\left(L^{2} \right)} ,
\end{equation}
whose proofs will be omitted.

Next, we estimate $E_{2} $ and $F_{2} $.
We can take $s_{3} \left(>0\right)$ satisfying the following system:
\begin{equation} \label{GrindEQ__3_78_}
\left\{\begin{array}{l} {\frac{1}{\hat{r}'} -\left(\sigma +1\right)\left(\frac{1}{r_{2} } -\frac{s_{3} }{n} \right)>\frac{b+s}{n} ,\;} \\ {0<s_{3} <s,} \\ {\frac{1}{r_{2} } >\frac{s_{3} }{n} .} \end{array}\right.
\end{equation}
In fact, we can see that the first equation in (3.78) is equivalent to
\begin{equation} \label{GrindEQ__3_79_}
\frac{n}{r_{2} } -\frac{n}{\sigma +1} \left(\frac{n-b-s}{n} -\frac{1}{\hat{r}} \right)<s_{3} .
\end{equation}
Hence the system (3.78) is equivalent to
\begin{equation} \label{GrindEQ__3_80_}
\max \left\{0,\;\frac{n}{r_{2} } -\frac{n}{\sigma +1} \left(\frac{n-b-s}{n} -\frac{1}{\hat{r}} \right)\right\}<s_{3} <\min \left\{s,\;\frac{n}{r_{2} } \right\}.
\end{equation}
In view of (3.68), we can see that
\[\frac{n}{r_{2} } -\frac{n}{\sigma +1} \left(\frac{n-b-s}{n} -\frac{1}{\hat{r}} \right)<\frac{n}{r_{2} } .\]
In view of (3.69), we can also see that $\frac{n}{r_{2} } -\frac{n}{\sigma +1} \left(\frac{n-b-s}{n} -\frac{1}{\hat{r}} \right)<s$ is equivalent to
\begin{equation} \label{GrindEQ__3_81_}
\sigma \left(n-2s\right)<4-2b.
\end{equation}
Since $\sigma<\sigma_{s}$, we have (3.81). Thus we can take $s_{3}\left(>0\right)$ satisfying (3.80).
Using the same argument as in the estimates of $C_{2} $ and $D_{2} $ in Lemma 3.5, we have
\begin{equation} \label{GrindEQ__3_82_}
E_{2} \lesssim\left\| u\right\| _{S\left(H^{s} \right)}^{\sigma +1} ,~ F_{2} \lesssim\left\| u\right\| _{S\left(H^{s} \right)}^{\sigma } \left\| v\right\| _{S\left(L^{2} \right)} ,
\end{equation}
whose proofs will be omitted. This completes the proof.
\end{proof}

\section{Proofs of main results}

In this section, we prove Theorem 1.3 and Theorem 1.7.

\begin{proof}[\textbf{Proof of Theorem 1.3.}]
The proof is similar to one of Theorem 1.4 of \cite{G17} and we only sketch the proof. We define
\[X=C\left(\R,\;H^{s} (\R^{n})\right)\bigcap L^{\gamma \left(r\right)} \left(\R,\;H_{r}^{s} (\R^{n})\right),\]
for any admissible pair $\left(\gamma \left(r\right),\;r\right)$, and
\[\left\| u\right\| _{X} =\left\| u\right\| _{S\left(H^{s} \right)}.\]
Let $M>0$ which will be chosen later. We define the complete metric space
\[D=\left\{u\in X:\;\left\| u\right\| _{X} \le M\right\},\; d\left(u,\;v\right)=\left\| u-v\right\| _{S\left(L^{2} \right)} .\]
We consider the mapping
\[T:\;u(t)\to S(t)u_{0} -i\int _{0}^{t}S(t-\tau )|x|^{-b} f\left(u\left(\tau \right)\right)d\tau  .\]
It follows from the Lemma 2.9 (Strichartz estimates) that
\begin{equation} \label{GrindEQ__4_1_}
\left\| Tu\right\| _{X} \le \left\| Tu\right\| _{S(\dot{H}^{s})} +\left\| Tu\right\| _{S\left(L^{2} \right)} \lesssim\left\| u_{0} \right\| _{H^{s} } +\left\| |x|^{-b} f\left(u\right)\right\| _{S'\left(L^{2} \right)} +\left\| |x|^{-b} f\left(u\right)\right\| _{S'(\dot{H}^{s})}.
\end{equation}
Using Lemma 3.5 and Lemma 3.6, we have
\begin{equation} \label{GrindEQ__4_2_}
\left\| |x|^{-b} f\left(u\right)\right\| _{S'\left(L^{2} \right)} \lesssim\left\| |x|^{-b} \left|u\right|^{\sigma } u\right\| _{S'\left(L^{2} \right)} \lesssim\left\| u\right\| _{S\left(H^{s} \right)}^{\sigma } \left\| u\right\| _{S\left(L^{2} \right)} ,
\end{equation}
\begin{equation} \label{GrindEQ__4_3_}
\left\| |x|^{-b} f\left(u\right)\right\| _{S'(\dot{H}^{s})} \lesssim\left\| u\right\| _{S\left(H^{s} \right)}^{\sigma +1} .
\end{equation}
In view of (4.1)--(4.3), we have
\begin{equation} \label{GrindEQ__4_4_}
\left\| Tu\right\| _{X} \le C\left( \left\| u_{0} \right\| _{H^{s} } +\left\| u\right\| _{S\left(H^{s} \right)}^{\sigma +1}\right) .
\end{equation}
Using the same argument as in Remark 2.6 of \cite{G17}, we have
\begin{equation} \label{GrindEQ__4_5_}
\left||x|^{-b} f\left(u\right)-|x|^{-b} f\left(v\right)\right|\lesssim|x|^{-b} \left(\left|u\right|^{\sigma } +\left|v\right|^{\sigma } \right)\left|u-v\right|.
\end{equation}
Using Lemma 2.9 (Strichartz estimates), Lemma 3.5 and Lemma 3.6, we have
\begin{eqnarray}\begin{split} \label{GrindEQ__4_6_}
d\left(Tu,\;Tv\right)&\le C\left\| |x|^{-b} \left(\left|u\right|^{\sigma } +\left|v\right|^{\sigma } \right)\left|u-v\right|\right\| _{S'\left(L^{2} \right)} \;\\
&\le C\left(\left\| u\right\| _{S\left(H^{s} \right)}^{\sigma } +\left\| v\right\| _{S\left(H^{s} \right)}^{\sigma } \right)d\left(u,\;v\right).
\end{split}\end{eqnarray}
Put $M=2C\left\| u_{0} \right\| _{H^{s} } $ and $\delta =2\left(4C\right)^{-\frac{\sigma +1}{\sigma } } $. If $\left\| u_{0} \right\| _{H^{s} } \le \delta $, then we have $CM^{\sigma } \le \frac{1}{4} $. Hence it follows from (4.4) and (4.6) that $T:\left(D,d\right)\to \left(D,d\right)$ is a contraction mapping. So there is a unique global solution satisfying (1.8). This completes the proof.
\end{proof}
\begin{proof}[\textbf{Proof of Theorem 1.7.}]

The proof is standard and we only sketch the proof (see e.g. \cite{D19}).
Let $u$ be the global solution of (1.1) with initial data $u_{0} \in H^{s} $ given in Theorem 1.3.
We can see that (1.10) is equivalent to
\begin{equation} \label{GrindEQ__4_7_}
{\mathop{\lim }\limits_{t\to \pm \infty }} \left\| e^{-it\Delta } u(t)-u_{0}^{\pm } \right\| _{H^{s} (\R^{n})} =0.
\end{equation}
In other words, it suffices to show that $e^{-it\Delta } u(t)$ converges in $H^{s} $ as $t_{1} ,\;t_{2} \to \pm \infty $.

\noindent Let $0<t_{1} <t_{2} <+\infty $. By using Strichartz estimates, we have
\begin{eqnarray}\begin{split} \label{GrindEQ__4_8_}
&\left\| e^{-it_{2} \Delta } u\left(t_{2} \right)-e^{-it_{1} \Delta } u\left(t_{1} \right)\right\| _{H^{s} } =\left\| \int _{t_{1} }^{t_{2} }e^{-i\tau \Delta } |x|^{-b} f\left(u\left(\tau \right)\right)d\tau  \right\| _{H^{s} } \\
&~~~~~~~~~~~~~~~~~\lesssim\left\| |x|^{-b} f\left(u\right)\right\| _{S'\left(\left(t_{1} ,\;t_{2} \right),\;\dot{H}^{s} \right)} +\left\| |x|^{-b} f\left(u\right)\right\| _{S'\left(\left(t_{1} ,\;t_{2} \right),\;L^{2} \right)} .
\end{split}
\end{eqnarray}
It follows from Lemma 3.5 and Lemma 3.6 that
\begin{equation} \label{GrindEQ__4_9_}
\left\| |x|^{-b} f\left(u\right)\right\| _{S'(\dot{H}^{s})} +\left\| |x|^{-b} f\left(u\right)\right\| _{S'\left(L^{2} \right)} \lesssim\left\| u\right\| _{S\left(H^{s} \right)}^{\sigma +1} <\infty .
\end{equation}
(4.8) and (4.9) implies that
\[\left\| e^{-it_{2} \Delta } u\left(t_{2} \right)-e^{-it_{1} \Delta } u\left(t_{1} \right)\right\| _{H^{s} (\R^{n})} \to 0,~\textnormal{as}~t_{1}, \;t_{2} \to +\infty. \]
Thus the limit
\[u_{0}^{+} :={\mathop{\lim }\limits_{t\to +\infty }} e^{-it\Delta } u\left(t\right)\]
exits in $H^{s} (\R^{n})$. This shows the small data scattering for positive time, the one for negative time is treated similarly. This concludes the proof.
\end{proof}


\end{document}